\documentclass[]{amsart}
\usepackage{amssymb}
\usepackage{amsmath}
\usepackage{amsthm}
\usepackage[textwidth=16cm, textheight=24cm, marginratio=1:1]{geometry}

\usepackage[english]{babel}
\usepackage[utf8]{inputenc}
\usepackage[T1]{fontenc}
\usepackage{xcolor}
\usepackage{enumitem}
\usepackage[backref=page]{hyperref}
\numberwithin{equation}{section}
\newtheorem{theorem}{Theorem}[section]
\newtheorem{corollary}[theorem]{Corollary}
\newtheorem{lemma}[theorem]{Lemma}

\newtheorem{proposition}[theorem]{Proposition}
\theoremstyle{definition}

\newtheorem{example}[theorem]{Example}
\hypersetup{
    colorlinks,
    linkcolor={red!80!black},
    citecolor={blue!80!black},
    urlcolor={blue!80!black}
}
\usepackage{setspace}
\usepackage{tikz}
\usetikzlibrary{matrix,arrows, cd, calc}  
\DeclareMathOperator{\im}{im}
\DeclareMathOperator{\Spec}{Spec}
\DeclareMathOperator{\Hom}{Hom}
\DeclareMathOperator{\Ext}{Ext}
\DeclareMathOperator{\Sym}{Sym}
\DeclareMathOperator{\Hilb}{Hilb}
\DeclareMathOperator{\Quot}{Quot}
\DeclareMathOperator{\Gr}{Gr}%
\newcommand{\onto}{\twoheadrightarrow}
\newcommand{\into}{\hookrightarrow}

\newcommand{\kk}{\Bbbk}%

\newcommand{\OO}{\mathcal{O}}
\newcommand{\pts}{\mathrm{pts}}
\newcommand{\mm}{\mathfrak{m}}
\newcommand{\inn}{\operatorname{in}}
\newcommand{\BBname}{Bia{\l}ynicki-Birula}%
\newcommand{\Gmult}{\mathbb{G}_m}%

\begin{document}
\selectlanguage{english}
\title[Generically nonreduced components of Hilbert schemes]{Generically nonreduced components of Hilbert schemes on
fourfolds}
\author[Jelisiejew]{Joachim~Jelisiejew}
\address{Faculty of Mathematics, Informatics and Mechanics, University of Warsaw, Banacha 2, 02-097 Warsaw}
\email{j.jelisiejew@uw.edu.pl}
\thanks{University of Warsaw. Partially supported by
National Science Centre grant 2020/39/D/ST1/00132. The article is accepted in
the special volume of \emph{Rendiconti del seminario matematico di Torino} dedicated to Prof. Casnati.}
\dedicatory{To Gianfranco}
\maketitle

    \begin{abstract}
        We exhibit generically nonreduced components
        of the Hilbert scheme of at least $21$ points on a smooth
        variety of dimension at least four. The result was announced
        in~\cite{Jelisiejew__open_problems} and answers a
        question~\cite[Problem~3.8]{aimpl}. The method is similar to the one
        of~\cite[\S6]{jelisiejew_sivic}.
    \end{abstract}

    \section{Introduction}

        Let $X$ be a smooth quasi-projective variety over a field $\kk$.
        The Hilbert scheme of $d$ points $\Hilb_d(X)$ is a moduli space of central importance, with
        applications to combinatorics~\cite{Haiman_macdonald,
        haiman_factorial_conjecture}, algebra, enumerative
        geometry~\cite{Ricolfi__Modern_Enumerative_Geometry}, and
        classical algebraic geometry~\cite{Beauville__Hilbert_K3}. Many of the applications are limited to
        the case when $\dim X\leq 2$ as in this case $\Hilb_d(X)$ is
        smooth~\cite{fogarty}. For a good and gentle introduction to Hilbert schemes,
        see~\cite{Bertin__punctual_Hilbert_schemes}
        or~\cite[Chapter~18]{Miller_Sturmfels}. See
        also~\cite{Jelisiejew__open_problems} for a list of open problems.

        The possible singularities of $\bigsqcup_d\Hilb_d(X)$ for $\dim X \geq 3$ are only
        partially understood. A point $[Z]\in \Hilb_d(X)$ is smooth for every
        $Z\subseteq X$ which can be embedded into a smooth surface. As a very
        particular case, this implies that $\Hilb_{d}(X)$ is smooth for $d\leq
        3$. In contrast, the Hilbert scheme $\Hilb_{d}(X)$ is singular for every
        $d\geq 4$ and $\dim X \geq 3$, in fact for every $x\in X$, any degree
        $d$ subscheme
        $Z\subseteq V(\mm_x^2)$ gives a singular point~\cite[Cor~18.30]{Miller_Sturmfels}.

        The singularities in the case $\dim X = 3$ are constrained as the
        Hilbert scheme is a critical
        locus~\cite{Dimca_Szendroi__Hilbert_of_A3}. Understanding the
        singularities is a very active research area, see for
        example~\cite{Graffeo_Giovenzana_Lella, Ramkumar_Sammartano_parity,
        Kool_Jelisiejew_Schmierman, Rezaee__Conjectural_most_singular}.

        The singularities in the case $\dim X \geq 16$ can be almost
        arbitrary: the Hilbert scheme satisfies Murphy's Law
        up to retraction, see~\cite{Jelisiejew__Pathologies}.
        For important singularity types, such as nonreduced ones, sharper bounds on $\dim X$ are known.
        Szachniewicz~\cite{Szachniewicz} proved that $\Hilb_{d}(\mathbb{A}^6)$
        is nonreduced for every $d\geq 13$; it has an embedded component.
        See~\cite{erman_Murphys_law_for_punctual_Hilb, Schmiermann} for some
        results in similar direction on fixed loci.

        One instance where \emph{up to retraction} cannot be ignored is when we consider generic
        nonreducedness.
        In particular, the results above do not prove that the
        Hilbert scheme has any generically nonreduced components.
        Proving that
        such components do exist and already in codimension four is the main aim of the current article.

        \subsection{Generic nonreducedness}
        We work over a field of characteristic zero, in particular over a
        perfect field. An irreducible component of a
        finite type $\kk$-scheme is either \emph{generically smooth}, that is,
        its general point is smooth, or \emph{generically nonreduced} which
        means that every point is nonreduced.

        The problem is that generic nonreducedness does not propagate along
        retractions.
        For example, consider
        \[
            \frac{\kk[\![y]\!]}{(y^2)}\into \frac{\kk[\![x, y]\!]}{(xy, y^2)}.
        \]
        The source of this map is generically nonreduced, while the target is
        generically reduced. Geometrically speaking, the above map comes from
        a retraction of $V(xy, y^2)\subseteq \mathbb{A}^2$ onto $V(y^2)
        \subseteq \mathbb{A}^1$ by contracting the $x$ axis:
        
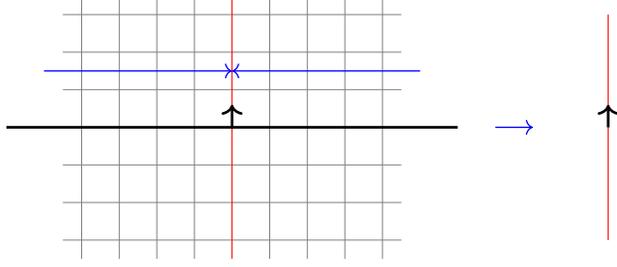
\[
            \begin{tikzpicture}
                \draw[color=gray, very thin] (-2.25, -1.5) -- (2.25, -1.5);
                \draw[color=gray, very thin] (-2.25, -1.0) -- (2.25, -1.0);
                \draw[color=gray, very thin] (-2.25, -0.5) -- (2.25, -0.5);
                \draw[color=gray, very thin] (-2.25, 1.5) -- (2.25, 1.5);
                \draw[color=gray, very thin] (-2.25, 1.0) -- (2.25, 1.0);
                \draw[color=gray, very thin] (-2.25, 0.5) -- (2.25, 0.5);
                \draw[color=gray, very thin] (-2, -1.75) -- (-2, 1.75);
                \draw[color=gray, very thin] (-1.5, -1.75) -- (-1.5, 1.75);
                \draw[color=gray, very thin] (-1, -1.75) -- (-1, 1.75);
                \draw[color=gray, very thin] (-0.5, -1.75) -- (-0.5, 1.75);
                \draw[color=red] (0, -1.75) -- (0, 1.75);
                \draw[color=gray, very thin] (0.5, -1.75) -- (0.5, 1.75);
                \draw[color=gray, very thin] (1, -1.75) -- (1, 1.75);
                \draw[color=gray, very thin] (1.5, -1.75) -- (1.5, 1.75);
                \draw[color=gray, very thin] (2, -1.75) -- (2, 1.75);
                \draw[line width=1pt] (-3, 0) -- (3, 0);
                \draw[line width=1pt, ->] (0, 0) -- (0, 0.3);
                \draw[red] (5, -1.5) -- (5, 1.5);
                \draw[line width=1pt, ->] (5, 0) -- (5, 0.3);
                \draw[line width=0.5pt, color=blue, ->] (-2.5, 0.75) -- (0,
                0.75);
                \draw[line width=0.5pt, color=blue, ->] (2.5, 0.75) -- (0,
                0.75);
                \draw[->, color=blue, line width=0.5pt] (3.5, 0) -- (4, 0);
            \end{tikzpicture}
        \]
        Therefore, from~\cite{Jelisiejew__Pathologies} it does not follow
        that $\Hilb_d(\mathbb{A}^{16})$ admits generically nonreduced
        components. Neither it follows from subsequent paper of
        Szachniewicz~\cite{Szachniewicz}.
        In contrast, in the paper~\cite{jelisiejew_sivic} the authors show that
        $\Quot_{8}(\mathbb{A}^4)$ has a generically nonreduced component.

        The aim of the present note is to apply the method of
        Jelisiejew-\v{S}ivic to the case of the Hilbert scheme and show the
        following theorem, which resolves~\cite[Problem~3.8]{aimpl}. Let
        $\mathcal{L}_{H}\subseteq \Hilb_{\pts}(\mathbb{A}^n)$ denote the locus
        of $[Z]$ such that $Z = \Spec(A)$ is an irreducible scheme
        corresponding to the local algebra $A$ with Hilbert function $H_A = H$.

        \begin{theorem}\label{ref:mainthm}
            Let $\kk$ be a field of characteristic zero and let $H =
            (1,4,10,s)$ for $s\in \{6,7,8,9\}$. Then $\mathcal{L}_{H}
            \subseteq \Hilb_{15+s}(\mathbb{A}^4)$ is an
            irreducible component, and this component, with the scheme
            structure inherited from the Hilbert scheme, is generically
            nonreduced. Therefore, the Hilbert scheme
            $\Hilb_{d}(\mathbb{A}^4)$ admits generically nonreduced components
            for all $d\geq21$.
        \end{theorem}
        Prior to Theorem~\ref{ref:mainthm} it was not known whether
        $\Hilb_d(\mathbb{A}^4)$ or $\Hilb_d(\mathbb{A}^5)$ are \emph{reduced}
        for all $d$.
        It remains an open question whether $\Hilb_d(\mathbb{A}^3)$ is nonreduced
        for $d$ high enough and whether this scheme has generically nonreduced
        components, see~\cite[Problem~XIV]{Jelisiejew__open_problems}.

        There are three main steps of the argument. First, the locus
        $\mathcal{L}_H$ is closed for the functions $H$ as in theorem. Moreover,
        it is contained in a dominant \BBname{} cell, which
        implies that on an open subset $U\subseteq\Hilb_{15+s}(\mathbb{A}^4)$
        containing $\mathcal{L}_H$, the Hilbert scheme admits
        a retraction $\pi\colon U\to U^{\Gmult}$ which maps any point
        $[\Spec(S/I)]\in U$ to $\Spec(S/\inn(I))$, where $S = \kk[x_1, \ldots,
        x_4]$ and $\inn(I)$ is the ideal of top degree forms.

        Second, primary obstruction yields quadratic equations
        for the fibre
        \[
            \pi^{-1}([Z])\subseteq \left(T_{\Hilb_{15+s}(\mathbb{A}^4),
        [Z]}\right)_{<0}.
        \]

        Third, for a chosen $[Z_0]\in \mathcal{L}_H$ we computer-check using \emph{Macaulay2} that
        the quadrics alone cut out a $4$-dimensional scheme in the affine
        space $\left(T_{\Hilb_{15+s}(\mathbb{A}^4), [Z_0]}\right)_{<0}$. It follows that $\dim
        \pi^{-1}([Z_0])\leq 4$.
        The fibre $\pi^{-1}([Z_0])$ is a cone and has a translation action by
        $\mathbb{A}^4$, so the fibre is equal to $\{Z_0 + v\ |\ v\in
            \mathbb{A}^4\}$ as a set
        and hence $\mathcal{L}_H$ contains an open neighbourhood of $[Z_0]$,
        so this locus
        is a component. A syzygetic argument shows that
        the containment $T_{\mathcal{L}_H, [Z]} \subseteq
        T_{\Hilb_{15+s}(\mathbb{A}^4), [Z]}$ is strict for every $[Z]\in
        \mathcal{L}_H$, hence $\mathcal{L}_H$ cannot be generically reduced.

        \subsection{Open questions and possible generalizations}

            Consider now $\Hilb_d(\mathbb{A}^n)$ and the unique very
            compressed Hilbert function $H = H_{n,d}$ given by the condition
            that there exists a $\delta$ such that
            \[
                H_{n,d}(i) = \begin{cases}
                    \binom{n+i-1}{i} = \dim \kk[x_1, \ldots ,x_n]_i & \mbox{
                    for } i < \delta\\
                    0 & \mbox{ for } i > \delta\\
                    d - \sum_{i=0}^{\delta-1} \binom{n+i-1}{i} & \mbox{ for }
                    i =\delta
                \end{cases}
            \]
            The locus $\mathcal{L}_H\subseteq \Hilb_d(\mathbb{A}^n)$ is irreducible and closed also in this more general
            case.  We then have three possibilities for a general $[Z]\in
            \mathcal{L}_H$:
            \begin{enumerate}
                \item[SMOOTH]\label{it:smooth} the scheme $[Z]$ has only trivial negative tangents, so
                    $\mathcal{L}_H$ is a component and $[Z]$ is a smooth point
                    of $\Hilb_d(\mathbb{A}^n)$ on this component,
                \item[DEFORMS]\label{it:deforms} the scheme $[Z]$ has nontrivial negative tangents and
                    some of them ``integrate'', that is,
                    the fibre $\pi^{-1}([Z])$, which is a cone, contains more points than just
                    $\mathbb{A}^n$. In this case $\mathcal{L}_H$ is not an
                    irreducible component and without additional information we cannot say much about
                    whether its points are reduced in $\Hilb_d(\mathbb{A}^n)$.
                \item[NONRED]\label{it:nonred} the scheme $[Z]$ has nontrivial negative tangents and
                    $\pi^{-1}([Z])$ is, as a topological space, equal to
                    $\mathbb{A}^n$. In this case $\mathcal{L}_H$ is a
                    generically nonreduced component.
            \end{enumerate}

            \newcommand{\refsmooth}{\hyperref[it:smooth]{SMOOTH}}
            \newcommand{\refdeforms}{\hyperref[it:deforms]{DEFORMS}}
            \newcommand{\refnonred}{\hyperref[it:nonred]{NONRED}}

            \begin{example}
                By~\cite{CEVV} the case~\refsmooth{} occurs for example for $H =
                (1,4,3)$. The case~\refdeforms{} occurs for example for $H =
                (1,4,4)$. The case~\refnonred{} occurs for $H$ as in
                Theorem~\ref{ref:mainthm}.
            \end{example}
            We stress that above we look at a general point of
            $\mathcal{L}_H$. This makes a difference: for example for $H =
            (1,6,6)$ the case~\refsmooth{} occurs, so $\mathcal{L}_H$ is a
            generically smooth component, however
            Szachniewicz~\cite{Szachniewicz} found an embedded component of
            $\Hilb_{13}(\mathbb{A}^6)$ inside $\mathcal{L}_H$. It is a
            completely open problem to understand whether having an embedded
            component is typical or exceptional for $\mathcal{L}_H$ which fall
            into the~\refsmooth{} case.

            One motivation to discuss the more general situation is the case
            $n=3$, the Hilbert scheme of $\mathbb{A}^3$. Taking $d = 96$ and
            $H = H_{3, 96}$ we
            get that $\mathcal{L}_H$ is too big to fit in the smoothable
            component of $\Hilb_{96}(\mathbb{A}^3)$,
            see~\cite{iarrobino_reducibility}. A syzygetic argument, see
            Lemma~\ref{ref:tangentAtVeryCompressed:lemma} below, also shows
            that the case~\refsmooth{} cannot hold. Moreover, it is known
            that the fibre $\pi^{-1}([Z])$ for a general $[Z]$ is cut out by
            quadrics only. Actually, this holds whenever
            $(T^2_{[Z]})_{<-2} = 0$, where $T^2_{[Z]} \subseteq \Ext^1(I_Z,
            \OO_Z)$ is the Schessinger's functor, see~\cite[Chapter 3]{HarDeform}.
            It is possible to obtain the quadrics explicitly using
            \emph{Macaulay2}. However, the Gr\"obner basis computation
            necessary for determining $\dim \pi^{-1}([Z])$ is out of reach, at
            least using standard algorithms. We warn the reader that it is not
            clear, even intuitively, whether we should expect \refnonred{}
            or \refdeforms{} in this case, since it may be that
            $\mathcal{L}_H$ lies in the closure of a compressed (not very
            compressed) component similar to the ones discussed
            in~\cite{iarrobino_compressed_artin}.

            The question about $H = (1,4,10,s)$ in~\cite{aimpl} is also formulated
            for $s=10$. In this case one could try the approach above, however
            there are $50$ negative tangents (see
            Lemma~\ref{ref:tangentAtVeryCompressed:lemma} below) and the approach is
            infeasible on our hardware. Of course, perhaps this is
            only a question of computational cost, however we prefer to leave
            the case $s=10$ open, in the hope that it will stimulate further
            progress on understanding the Yoneda multiplication in
            $\Ext^{\bullet}(\OO_Z, \OO_Z)$ and in particular the primary obstruction.

            \section{Acknowledgements}

                We would like to thanks the Editors, in particular Ada
                Boralevi, for making this volume possible. We also thank the
                organizers of the GC workshop: Ada, Enrico, Paolo, and Roberto
                for this serene and productive time. The comments of the
                anonymous referee were helpful as well.

            \section{Preliminaries}

            We work over a field $\kk$ of characteristic zero. The
            characteristic assumption will be used mostly for justifying the
            computations (we believe that the result holds for most
            characteristics). Let $S =
            \kk[x_1, \ldots ,x_n]$ be a polynomial ring and $\mathbb{A}^n =
            \Spec(S)$.  For a subscheme $Z\subseteq \mathbb{A}^n$ we denote by
            $I_Z$ its ideal and by $\OO_Z = S/I_Z$ its coordinate ring.
                \begin{proposition}
                    The tangent space to $[Z]\in \Hilb_d(\mathbb{A}^n)$ is
                    given by $\Hom_S(I_Z, \OO_Z)$. This space is canonically
                    isomorphic to $\Ext^1_S(\OO_Z, \OO_Z)$.
                \end{proposition}
                \begin{proof}
                    A self-contained proof for $\Hom_S(I_Z, \OO_Z)$ can be found
                    in~\cite{Stromme_Hilbert}; also the Ext functor naturally
                    appears there. The isomorphism
                    \[
                        \Hom_S(I_Z, \OO_Z)\to \Ext^1_S(\OO_Z,\OO_Z)
                    \]
                    follows from the long exact sequence obtained by applying
                    $\Hom_S(-,\OO_Z)$ to $0\to I_Z\to S\to \OO_Z\to 0$.
                \end{proof}
                Further in the paper, when discussing $\Hom_S(-,-)$ and $\Ext_S^{\bullet}(-,-)$, we drop the
                subscript $S$ from the notation.

                If $I_Z$ is presented as
                \[
                    S^{\oplus d_2}\to S^{\oplus d_1}\to I_Z\to 0,
                \]
                then $\Hom(I_Z, \OO_Z)$ is the kernel of the natural map
                $\Hom(S^{\oplus d_1}, \OO_Z)\to \Hom(S^{\oplus d_2}, \OO_Z)$.
                Computing this kernel is best performed with
                a computer.

                We propose one example, which is straightforward, but it will
                be important in the following. Recall that when $I_Z$ is
                graded, also the tangent space $\Hom_S(I_Z, \OO_Z)$ is graded
                with
                \[
                    \Hom_S(I_Z, \OO_Z)_i = \left\{ \varphi\colon I_Z\to \OO_Z\
                        |\ \varphi((I_Z)_j)\subseteq (\OO_Z)_{i+j}\mbox{ for
                        all } j\right\}.
                \]
                \begin{lemma}\label{ref:tangentAtVeryCompressed:lemma}
                    Suppose that $[Z]\in \Hilb_{15+s}(\mathbb{A}^4)$ is given by a
                    \emph{homogeneous} ideal $I_Z$ and that $\OO_Z$ is very
                    compressed with Hilbert function $(1,4,10,s)$. Suppose
                    further that $I_Z$ is generated by cubics.
                    Then
                    \[
                        \dim \Hom(I_Z, \OO_Z)_0 = (20-s)s\quad \mbox{and}\quad
                        \dim \Hom(I_Z, \OO_Z)_{-1} \geq 4s^2 - 55s + 200.
                    \]
                \end{lemma}
                \begin{proof}
                    By assumption, the presentation of $I_Z$ is
                    \[
                        S(-5)^{\beta} \oplus S(-4)^{4\cdot (20-s)-35} \to
                        S(-3)^{20-s}\to I_Z\to 0.
                    \]
                    Let us first look at degree zero.
                    If we consider the full linear space $\Hom_{\kk}(I_Z, \OO_Z)_0$, then any relation
                    between generators of $I_Z$ is mapped to $(\OO_Z)_{\geq 4} =
                    0$, so $\Hom_{\kk}(I_Z, \OO_Z)_0 = \Hom(I_Z, \OO_Z)_0$.

                    Let us now look at degree one. By similar considerations,
                    for every linear map $\varphi\in \Hom_{\kk}(I_Z,
                    \OO_Z)_{-1}$, the image of $S(-5)^{\beta}$ is zero and the
                    image of $S(-4)^{4\cdot (20-s)-35}$ is contained in
                    the $s$-dimensional space $(\OO_Z)_3$. Thus, the relations
                    in the presentation yield at most $s\cdot \left( 4\cdot (20-s)-35
                    \right)$ linear-algebraic conditions on the images of
                    minimal homogeneous generators and so
                    \[
                        \dim \Hom(I_Z, \OO_Z)_{-1} \geq 10\cdot (20-s) - s\cdot \left( 4\cdot (20-s)-35
                            \right) = 4s^2 - 55s + 200,
                    \]
                    as claimed.
                \end{proof}

                \begin{proposition}[Very compressed
                    loci]\label{ref:verycompressed:prop}
                    Let $H$ be any very compressed Hilbert function and
                    $\delta$ be the largest index such that $H(\delta)\neq 0$. Then the
                    very compressed locus $\mathcal{L}_H$ is closed in
                    $\Hilb_d(\mathbb{A}^n)$, isomorphic to $\mathbb{A}^n \times
                    \Gr(H(\delta), \binom{n-1+\delta}{\delta})$ and has dimension
                    \[
                        n + \binom{n-1+\delta}{\delta} - H(\delta).
                    \]
                \end{proposition}
                \begin{proof}
                    A point $[Z]\in \Hilb_d(\mathbb{A}^n)$ lies in
                    $\mathcal{L}_H$ if and only if, first, the support of
                    $[Z]$ is a single point $z$ and, second, the ideal $I_Z$
                    is contained in $\mm_{z}^{\delta}$. The first condition is
                    closed and the second is closed provided that the first
                    one is satisfied. The description of $\mathcal{L}_H$ as a
                    product is immediate, see~\cite[Proposition~2.27]{Szachniewicz}.
                \end{proof}

                \subsection{\BBname{} decompositions}

                The general theory of \BBname{} decompositions is beautiful
                but quite complicated, see~\cite{jelisiejew_sienkiewicz__BB,
                jelisiejew_sienkiewicz__BB_poschar, Thaddeus__VGIT}.
                We would like to apply it to the standard scalar torus action
                on the Hilbert scheme. We will see below that in this special
                case things simplify considerably. Therefore, we gather below
                only the necessary facts and restrict to the affine case and
                to the positive \BBname{} decomposition, that is, when
                considering the limit at $t\to 0$.\footnote{Be aware that in
                    some articles by the author, notably~\cite{Jelisiejew__Elementary}, the
                    sign $X^+$ denotes the \emph{negative} \BBname{}
                decomposition, that is, the one coming from considering
                $\lim_{t\to \infty}$.}

                The following allows us to reduce to considering the affine
                case.
                \begin{proposition}[{\cite{Sumihiro},
                        \cite[Proposition~5.3(2)]{jelisiejew_sienkiewicz__BB}}]\label{ref:BBreduceToAffine:prop}
                    Suppose that $X$ is a quasi-projective scheme. Then there
                    is a open cover $\{U_i\}$ by affine $\Gmult$-stable
                    schemes. Moreover, for every such cover the \BBname{} decomposition $X^+$ of $X$ is
                    covered by the \BBname{} decompositions $U_i^+$ of $U_i$.
                \end{proposition}

                Next, a $\Gmult$-action on an affine scheme $\Spec(A)$ is the same
                as a $\mathbb{Z}$-grading on the algebra $A$. In this case,
                the \BBname{} decomposition can be characterised explicitly as
                follows.
                \begin{proposition}\label{ref:BBforAffine:prop}
                    Let $X = \Spec(A)$ be an affine scheme with an action of
                    $\Gmult$. Then, the positive \BBname{} decomposition
                            $X^+$ of $X$ is a closed subscheme $X^+$ given by the
                            ideal generated by $A_{<0}$. The fixed locus of $X$ is given by the ideal
                            generated by $\{A_{<0}\}\cup \{A_{>0}\}$.
                        The composition
                            \[
                                \frac{A}{A_{<0}\cdot A + A_{>0}\cdot A} \simeq
                                \frac{A_0}{(A_{<0}\cdot A)_{0}}\into \frac{A_{\geq 0}}{(A_{<0}\cdot A)_{\geq 0}}
                                \simeq \frac{A}{A_{<0}\cdot A}
                            \]
                            gives a
                            morphism $\pi\colon X^+\to X^{\Gmult}$.
                    The canonical closed embedding $s\colon X^{\Gmult}\to
                    X^+$ is a section of $\pi$. We obtain the following
                    diagram, where $\pi$ and $s$ are closed embeddings
                    \[
                        \begin{tikzcd}
                            X^+ \ar[r, "\theta", hook]\ar[d, "\pi"] & X\\
                            X^{\Gmult}\ar[u, shift left=2, "s", hook]
                        \end{tikzcd}
                    \]
                    Thus, for every $x\in X^{\Gmult}$, the fibre
                    $\pi^{-1}(x)$ is given by spectrum of an
                    $\mathbb{N}$-graded algebra $B  \simeq \frac{A}{A_{<0}\cdot
                    A+(\mm_{x})_0\cdot A}$, which satisfies $B_0
                    = \kk$.
                \end{proposition}
                \begin{proof}
                    See for
                    example~\cite[Proposition~4.5, Example~4.6]{jelisiejew_sienkiewicz__BB}.
                \end{proof}

                For a homogeneous maximal ideal $\mm$ in a $\mathbb{Z}$-graded
                ring $A$, the cotangent space at $[\mm]\in\Spec(A)$ is the
                subquotient $\mm/\mm^2$, so is also naturally
                graded. The tangent space at $[\mm]$ is also graded, the
                weights are opposite.
                \begin{example}\label{ex:tangentSpace}
                    In the setup of Proposition~\ref{ref:BBforAffine:prop},
                    take $x\in X^{\Gmult}(\kk)$. Then the cotangent
                    space at $x\in X^+$ is the non-negative part of the
                    cotangent space of $x\in X$. Dualising, we obtain that
                    \[
                        d\theta\colon T_{X^+, x}\to T_{X, x}
                    \]
                    identifies $T_{X^+, x}$ with $(T_{X, x})_{\leq 0}$, the
                    non-\textbf{positive} part of $T_{X, x}$.
                \end{example}

                The weights of the tangent space are crucial for comparing
                $X^+$ and $X$, as the following proposition says.
                \begin{proposition}[{\cite[Proposition~1.6]{jelisiejew_sienkiewicz__BB}}]\label{ref:openImmersion:prop}
                    Let $X$ be a separated scheme locally of finite type (for
                    example, this holds if $X$ is quasi-projective).
                    Assume that $x\in X^{\Gmult}(\kk)$ is such that
                    $d\theta_x$ is surjective (that is, an
                    isomorphism). Then up to restricting to a $\Gmult$-stable
                    affine neighbourhood of $x$ we can assume that $\theta$ is an
                    isomorphism.
                \end{proposition}

                \subsection{\BBname{} decomposition of the Hilbert scheme of
                    points}

                Let $\Gmult =
                \Spec(\kk[t^{\pm1}])$ be a one-dimensional torus and consider
                its action
                \[
                    \Gmult \times \mathbb{A}^n\to \mathbb{A}^n
                \]
                by rescaling: $\lambda\cdot (x_1, \ldots ,x_n) = (\lambda x_1, \ldots
                ,\lambda x_n)$ for every $\kk$-point $(x_1, \ldots ,x_n)\in
                \mathbb{A}^n(\kk)$ and $\lambda\in \kk^{\times} = \Gmult(\kk)$.
                For every closed subscheme $Z\subseteq \mathbb{A}^n$ and $\lambda\in
                \Gmult(\kk)$ we obtain a new closed subscheme $\lambda\cdot Z$
                given by the closed embedding
                \begin{equation}\label{eq:action}
                    \begin{tikzcd}
                        Z \ar[r, hook] & \mathbb{A}^n \ar[r, "\lambda \cdot",
                        "\simeq"'] & \mathbb{A}^n.
                    \end{tikzcd}
                \end{equation}
                When we view a point as a closed subscheme, both definitions
                agree. A subscheme $Z$ is a $\Gmult$-fixed point if and only
                if its ideal $I_Z$ is homogeneous.

                Construction~\eqref{eq:action} generalizes readily to the case
                when $Z$ is closed in $\mathbb{A}^n \times S$, for any scheme
                $S$. This yields an action $\Gmult\times
                \Hilb_d(\mathbb{A}^n)\to \Hilb_d(\mathbb{A}^n)$, which on
                $\kk$-points agrees with~\eqref{eq:action}.

                For any set $A$ of $d$ monomials in $S$, consider the locus
                $U_A\subseteq \Hilb_d(\mathbb{A}^n)$ which consists of $[Z]\in
                \Hilb_d(\mathbb{A}^n)$ such that $A$ spans $\OO_Z$.
                These loci are open and $\Gmult$-stable, hence the
                corresponding \BBname{} cells $U_A^+$ cover
                $(\Hilb_d(\mathbb{A}^n))^+$,
                see Proposition~\ref{ref:BBreduceToAffine:prop}. The
                loci above are important for the computational aspects,
                see for example~\cite{Lella_Roggero__functoriality_of_marked}.

                We would like now to understand when
                Proposition~\ref{ref:openImmersion:prop} can be applied in the
                case of Hilbert schemes, so we are interested in the weights
                on the tangent space.
                \begin{lemma}
                    Let $[Z]\in \Hilb_d(\mathbb{A}^n)$ be a $\Gmult$-fixed
                    point. Then $(T_{\Hilb_d(\mathbb{A}^n), [Z]})_{>0}$
                    vanishes if and only if $\OO_Z$ is very compressed.
                \end{lemma}
                \begin{proof}
                    Take $S = \kk[x_1, \ldots ,x_n]$.
                    Suppose first that $\OO_Z$ is very compressed.
                    Then there exists an $s$ such that $I_Z\subseteq
                    S_{\geq s}$ and $(\OO_{Z})_{>s} = 0$. A tangent at $Z$ of
                    strictly positive degree $i$ corresponds to a homomorphism
                    $\varphi\colon I_Z\to \OO_Z$ such that
                    $\varphi((I_Z)_j)\subseteq (\OO_Z)_{j+i}$. The source is
                    nonzero only for $j\geq s$, but for such a $j$ we have
                    $i+j>s$, so the target is zero. It follows that $\varphi =
                    0$.

                    Suppose now that $\OO_Z$ is not very compressed. This
                    implies that there exists an $s$ such that $I_s \neq 0$
                    and $(\OO_{Z})_{\geq s+1} \neq 0$. Pick a set of minimal
                    generators of $I_Z$ and let $g$ be an element of lowest
                    degree. Pick a socle element $h\in \OO_Z$ of highest
                    degree. Then $\deg(h) \geq s+1 > s \geq \deg(g)$. There
                    exists a homomorphism $\varphi\colon I_Z\to \OO_Z$ which
                    satisfies $\varphi(g) = h$ and sends all other
                    minimal generators to zero. It follows that $\varphi$ is
                    homogeneous of strictly positive degree, equal to
                    $\deg(h)-\deg(g)$.
                \end{proof}

                \subsection{Primary obstructions}

                    Primary obstructions govern the order two part of
                    deformation theory and can be computed explicitly. We
                    discuss them below.

                    Consider two tangent vectors at a point $[Z]\in
                    \Hilb_d(\mathbb{A}^n)$. They yield maps
                    \[
                        \varphi_i\colon \Spec\left(
                        \frac{\kk[\varepsilon_i]}{(\varepsilon_i^2)}
                        \right)\to \Hilb_d(\mathbb{A}^n),
                    \]
                    for $i=1,2$ and two elements
                    \[
                        \varphi_1, \varphi_2\in T_{\Hilb_{d}(\mathbb{A}^n),
                        [Z]} \simeq \Ext^1\left( \OO_Z, \OO_Z \right).
                    \]
                    The two tangent vectors span an at most
                    $2$-dimensional space and the corresponding morphism is
                    \[
                        \varphi_{12}\colon \Spec\left(
                        \frac{\kk[\varepsilon_1,
                        \varepsilon_2]}{(\varepsilon_1, \varepsilon_2)^2}
                        \right)\to \Hilb_d(\mathbb{A}^n),
                    \]
                    which restricts to $\varphi_1$, $\varphi_2$ in the natural
                    way. We may ask when $\varphi_{12}$ does extend to a map
                    $\widetilde{\varphi_{12}}$
                    from $\Spec\left( \kk[\varepsilon_1,
                    \varepsilon_2]/(\varepsilon_1^2, \varepsilon_2^2)
                    \right)$. Deformation
                    theory~\cite[Chapter~5]{fantechi_et_al_fundamental_ag}
                    implies that an extension exists if and only if an
                    obstruction
                    \[
                        ob_{\varphi_{12}, \widetilde{\varphi}_{12}}\in
                        \Ext^2(\OO_Z, \OO_Z)
                    \]
                    vanishes. The key observation is that we can describe
                    the obstruction explicitly.
                    \begin{theorem}[{\cite[Theorem~4.18]{jelisiejew_sivic}}]\label{ref:primary:thm}
                        The obstruction $ob_{\varphi_{12},
                        \widetilde{\varphi}_{12}}$ is equal to
                        \[
                            \varphi_1\circ \varphi_2 + \varphi_2 \circ \varphi_1,
                        \]
                        where $\circ$
                        denotes Yoneda's multiplication in $\Ext^{\bullet}(\OO_Z,
                        \OO_Z)$ applied to $\varphi_1, \varphi_2\in
                        \Ext^1(\OO_Z, \OO_Z)$.
                    \end{theorem}
                    \noindent Let $\mu\colon \Sym^2\Ext^1(\OO_Z, \OO_Z)\to \Ext^2(\OO_Z, \OO_Z)$
                    be given by
                    \[
                        \mu(\varphi_1\cdot \varphi_2) := \varphi_1\circ \varphi_2 + \varphi_2 \circ \varphi_1.
                    \]
                    Consider its transpose
                    \[
                        \mu^{\vee}\colon \Ext^2(\OO_Z, \OO_Z)^{\vee}\to
                        \left(\Sym^2\Ext^1(\OO_Z, \OO_Z)\right)^{\vee} \simeq
                        \Sym^2\left(\Ext^1(\OO_Z, \OO_Z)^{\vee}\right).
                    \]
                    As explained in~\cite[\S4.2]{jelisiejew_sivic},
                    Theorem~\ref{ref:primary:thm} yields the following
                    corollary. Recall that $\Ext^1(\OO_Z, \OO_Z)^{\vee}$ is
                    the cotangent space at $[Z]\in \Hilb_d(\mathbb{A}^n)$.
                    \begin{corollary}\label{ref:primaryEquations:cor}
                        Consider the complete local ring
                        $(\hat{\OO}_{\Hilb_d(\mathbb{A}^n), [Z]},
                        \mm_{[Z]})$. Its truncation to second order satisfies
                        \[
                            \frac{\hat{\OO}_{\Hilb_d(\mathbb{A}^n),
                        [Z]}}{\mm_{[Z]}^3} \simeq \frac{\Sym^{\bullet} \Ext^1(\OO_Z,
                        \OO_Z)^{\vee}}{\im\left( \mu^{\vee}\colon \Ext^2(\OO_Z,
                        \OO_Z)^{\vee} \to \Sym^2\Ext^1(\OO_Z,
                        \OO_Z)^{\vee}\right) + \left( \Ext^1(\OO_Z,
                        \OO_Z)^{\vee} \right)^3}.
                        \]
                    \end{corollary}

                    \subsection{Primary obstructions and \BBname{}
                        decompositions}

                    As written, Corollary~\ref{ref:primaryEquations:cor} does
                    not directly involve the dimension of the local ring.
                    Moreover, we would like to apply it for the fibre of the
                    \BBname{} decomposition. Both subtleties ``cancel out'':
                    restriction to the fibre gives us an $\mathbb{N}$-grading
                    which allows to pass from the third neighbourhood to the
                    full complete local ring.

                    \begin{proposition}\label{ref:BBfiber:prop}
                        Let $[Z]\in \Hilb_d(\mathbb{A}^n)$ be a $\Gmult$-fixed
                        $\kk$-point and consider its \BBname{} fibre $\Spec(A)$.
                        Assume that the subspace
                        \[
                            \left(T_{\Hilb_d(\mathbb{A}^n),
                        [Z]}\right)_{\leq -2}
                        \]
                        is zero.
                        Then there is a surjection of graded algebras
                        \[
                            \frac{\Sym^{\bullet} \left(\Ext^1(\OO_Z,
                                \OO_Z)_{<0}\right)^{\vee}}{\im(\mu^{\vee}\colon
                                \left(\Ext^2(\OO_Z,
                                \OO_Z)_{<0}\right)^{\vee}
                                \to \Sym^2\left(\Ext^1(\OO_Z,
                                \OO_Z)_{<0}\right)^{\vee}
                            )} \onto A.
                        \]
                    \end{proposition}
                    \begin{proof}
                        By Proposition~\ref{ref:BBreduceToAffine:prop} we already know that there is an open
                        $\Gmult$-stable neighbourhood $[Z]\in U\subseteq
                        \Hilb_d(\mathbb{A}^n)$ such that the \BBname{} fibre
                        $\pi^{-1}([Z])$ is contained in $U^+$. Using
                        Proposition~\ref{ref:BBforAffine:prop} we conclude
                        that $\pi^{-1}([Z])$ is a spectrum of an
                        $\mathbb{N}$-graded algebra $B$ with $B_0 = \kk$.

                        We now employ Corollary~\ref{ref:primaryEquations:cor}.
                        To make the notation lighter, we put $E :=
                        \Ext^1(\OO_Z, \OO_Z)$.
                        The complete local ring $\hat{B}$ of $[Z]\in \Spec(B)$ is a quotient of the
                        complete local ring of $[Z]\in \Hilb_d(\mathbb{A}^n)$. Hence,
                        also the truncation $\hat{B}/\mm^3$ is a quotient of
                        the truncation of
                        \[
                            \frac{\hat{\OO}_{\Hilb_d(\mathbb{A}^n),[Z]}}{\mm^3}
                             \simeq  \frac{\Sym^{\bullet} E^{\vee}}{\im\left( \mu^{\vee}\colon \Ext^2(\OO_Z,
                            \OO_Z)^{\vee} \to \Sym^2 E^{\vee}\right) + \left( E^{\vee} \right)^3}.
                        \]
                        The cotangent space of $[Z]\in \Spec(B)$ has no
                        nonpositive weights, so $\hat{B}/\mm^3$ is in fact a
                        quotient of
                        \[
                            \frac{\Sym^{\bullet} \left(E_{<0}\right)^{\vee}}{\im\left( \mu^{\vee}\colon
                            \left(\Ext^2(\OO_Z, \OO_Z)_{<0}\right)^{\vee} \to
                            \Sym^2\left(E_{<0}\right)^{\vee}\right) + \left(
                            \left(E_{<0}\right)^{\vee} \right)^3}
                        \]
                        We now lift this from infinitesimal second order to a
                        more global situation using the $\mathbb{N}$-grading.
                        Consider the map of graded rings
                        \[
                            p\colon \Sym^{\bullet}E_{<0}\to B.
                        \]
                        By Example~\ref{ex:tangentSpace},
                        this map is an isomorphism on cotangent spaces. Since
                        $B$ is $\mathbb{N}$-graded with $B_0 = \kk$, the map
                        $p$ is a surjection, by induction on the degree.
                        Moreover, again thanks to the grading and to the fact
                        that all $E_{<0} = E_{-1}$, we obtain
                        an isomorphism $\hat{B}/\mm^3 \simeq B/B_{\geq 3}$.

                        We summarize the situation on
                        a commutative diagram
                        \[
                            \begin{tikzcd}
                                \Sym^{\bullet}E_{<0} \ar[r, two heads]\ar[d,
                                two heads] & B\ar[d]\ar[rd] \\
                            \frac{\Sym^{\bullet} \left(E_{<0}\right)^{\vee}}{\im\left( \mu^{\vee}\colon
                            \left(\Ext^2(\OO_Z, \OO_Z)_{<0}\right)^{\vee} \to
                            \Sym^2\left(E_{<0}\right)^{\vee}\right) + \left(
                            \left(E_{<0}\right)^{\vee} \right)^3} \ar[r,
                            "\simeq"] & \frac{\hat{B}}{\mm^3} \ar[r, "\simeq"]
                            &
                            \frac{B}{B_{\geq 3}}
                            \end{tikzcd}
                        \]
                        This shows that the image $p(\im \mu^{\vee})$ in $B_2$ is zero.
                        The claim follows.
                    \end{proof}

                \subsection{Computational input}

                    As mentioned in the introduction, currently there is not
                    enough knowledge about the Ext algebra to perform a
                    conceptual analysis of the primary obstruction. In this
                    section we include a somewhat brute-force check of
                    specific examples.

                    The package \emph{MatricesAndQuot} is available as an
                    auxiliary file for the arXiv version
                    of~\cite{jelisiejew_sivic}. Needless to say, many
                    alternatives exist, in particular the computation can be
                    performed using Ilten's \emph{VersalDeformations}
                    package~\cite{Ilten} or Lella's
                    \emph{HilbertAndQuotSchemesOfPoints.m2}
                    package~\cite{Lella_package}
                    or the framework~\cite{bertone_cioffi_roggero__smoothable_1771}.

                    \begin{proposition}[Key computational
                        output]\label{ref:fibredim:prop}
                        Let $H = (1,4,10,s)$
                        For every $s = 6,7,8,9$ there is an example of
                        $[Z]\in \mathcal{L}_H\subseteq
                        \Hilb_{15+s}(\mathbb{A}^4)$ such that the algebra
                        \[
                            \frac{\Sym^{\bullet} \left(\Ext^1(\OO_Z,
                                \OO_Z)_{<0}\right)^{\vee}}{\im(\mu^{\vee}\colon
                                \left(\Ext^2(\OO_Z,
                                \OO_Z)_{<0}\right)^{\vee}
                                \to \Sym^2\left(\Ext^1(\OO_Z,
                                \OO_Z)_{<0}\right)^{\vee})}
                            \]
                            is $4$-dimensional.
                    \end{proposition}
                    \begin{proof}
                        This is an explicit \emph{Macaulay2} computation. See
                        code below. We work over characteristic $17$ for
                        efficiency. By semicontinuity the same result (for the
                        ideals given by obvious lifts of generators) is true
                        in characteristic zero. See~\cite[(10.12)-(10.13)]{Kleiman_Kleppe} for
                        a detailed discussion of this method.
                    \begin{verbatim}
S = (ZZ/17)[x_1 .. x_4];
loadPackage("MatricesAndQuot", Reload=>true);
I9 = ideal(x_2*x_3^2, x_2^2*x_3+x_1*x_4^2+x_4^3,
x_1^2*x_2+x_1*x_3^2+x_2*x_4^2,
x_1^3+x_2^2*x_4+x_2*x_4^2, x_3*x_4^2, x_1^2*x_4,
x_1*x_2*x_4+x_3^2*x_4+x_1*x_4^2, x_1*x_2*x_3+x_3^3+x_1^2*x_4,
x_2^3+x_2*x_3*x_4+x_4^3, x_1^3+x_2^3+x_3^2*x_4,
x_1^2*x_3+x_1*x_2*x_4+x_2*x_4^2);
assert(degree I9 == 24); -- case (1,4,10,9)
assert(dim primaryObstruction(S^1/I9) == 4);
I8 = I9 + ideal(x_1*x_2^2);
assert(degree I8 == 23); -- case (1,4,10,8)
assert(dim primaryObstruction(S^1/I8) == 4);
I7 = I8 + ideal(x_1*x_3^2);
assert(degree I7 == 22); -- case (1,4,10,7)
assert(dim primaryObstruction(S^1/I7) == 4);
I6 = I7 + ideal(x_1^2*x_3);
assert(degree I6 == 21); -- case (1,4,10,6)
assert(dim primaryObstruction(S^1/I6) == 4);
\end{verbatim}
                    The total computation time is a few minutes, the case
                    \texttt{I9}
                    takes most of it.
                    \end{proof}

                \subsection{Proof of Theorem~\ref{ref:mainthm}}

                    We proceed to the proof of our main theorem. Recall that
                    we consider Hilbert functions $H = (1,4,10,s)$ for $s\in
                    \{6,7,8,9\}$.

                    \begin{proof}[Proof of Theorem~\ref{ref:mainthm}]
                        We follow the strategy outlined in the introduction.
                        Fix an $s$, let $d = 1+4+10+s$, and pick a point $[Z]\in \mathcal{L}_H$ as
                        in Proposition~\ref{ref:fibredim:prop}.
                        The fibre $\pi^{-1}([Z])$ is connected as it is a
                        cone.  The group scheme $(\mathbb{A}^4, +)$ acts on
                        the fibre $\pi^{-1}([Z])$ by translations. From this
                        and from $\dim \pi^{-1}([Z]) = 4$ it follows that the fibre is,
                        as a set, equal to the $(\mathbb{A}^4, +)$-orbit of
                        the cone point $[Z]$. By semicontinuity of fibre
                        dimensions, the same holds for fibres near $[Z]$.
                        It follows that, as a set, $\mathcal{L}_H$ contains an
                        open neighbourhood of $[Z]$.

                        From Proposition~\ref{ref:verycompressed:prop}
                        and Lemma~\ref{ref:tangentAtVeryCompressed:lemma} it follows that
                        for every point $[Z']\in \mathcal{L}_H$, the tangent
                        spaces to $\Hilb_{d}(\mathbb{A}^4)$ and
                        $\mathcal{L}_H$ differ already in degree $-1$. This can happen only if
                        the component of the Hilbert scheme that topologically
                        is equal to $\mathcal{L}_H$ has no smooth points, so
                        it is generically nonreduced.

                        To obtain generically nonreduced components of
                        $\Hilb_{d}(\mathbb{A}^4)$ for $d\geq 21$
                        consider a scheme $Z$ as above for $d=21$ and enlarge
                        it to a scheme
                        \[
                            Z\sqcup \bigsqcup_{d-21} \Spec(\kk)
                        \]
                        embedded (in any way) into $\mathbb{A}^4$.
                    \end{proof}

\newcommand{\etalchar}[1]{$^{#1}$}

\end{document}